\documentclass[11pt,a4paper]{amsart}
\usepackage{fullpage}

\usepackage{scalerel,amssymb}
\usepackage{amsthm,aliascnt}
\usepackage{nicefrac}
\usepackage{amsmath}
\usepackage[utf8]{inputenc}
\usepackage[T1]{fontenc}
\usepackage[british]{babel}

\usepackage{cite}

\usepackage{xcolor}
\usepackage[breaklinks]{hyperref}
\usepackage[capitalize,nameinlink,noabbrev]{cleveref}
\usepackage{breakurl}

\definecolor{blackred}{RGB}{183, 24, 82}
\definecolor{lgreen}{rgb}{0.0, 0.48, 0.0}
\definecolor{lpurple}{rgb}{0.48, 0.0, 0.48}
\definecolor{bblue}{rgb}{0.2, 0.4, 0.8}
\hypersetup{linktocpage,
            colorlinks=true,
            linkcolor=lgreen,
            citecolor=lpurple,
            linktoc=true}
\makeatletter
\renewcommand{\tocsection}[3]{%
  \indentlabel{\@ifnotempty{#2}{\bfseries\ignorespaces#1 #2\quad}}\bfseries#3}
\renewcommand{\tocsubsection}[3]{%
  \indentlabel{\@ifnotempty{#2}{\ignorespaces#1 #2\quad}}#3}
\makeatother

\usepackage{floatrow}
\usepackage{float}
\usepackage{algorithm,algorithmic}
\Crefname{ALC@unique}{Line}{Lines}

\usepackage{graphicx}
\usepackage{epstopdf}
\usepackage{float}
\graphicspath{{./images/}}
\usepackage{caption, subcaption}

\usepackage{enumerate}
\usepackage{tabu}
\usepackage{multirow}

\usepackage{etoolbox}

\usepackage{catoptions}
\makeatletter

\def\Autoref#1{%
  \begingroup
  \edef\reserved@a{\cpttrimspaces{#1}}%
  \ifcsndefTF{r@#1}{%
    \xaftercsname{\expandafter\testreftype\@fourthoffive}
      {r@\reserved@a}.\\{#1}%
  }{%
    \ref{#1}%
  }%
  \endgroup
}
\def\testreftype#1.#2\\#3{%
  \ifcsndefTF{#1autorefname}{%
    \def\reserved@a##1##2\@nil{%
      \uppercase{\def\ref@name{##1}}%
      \csn@edef{#1autorefname}{\ref@name##2}%
      \autoref{#3}%
    }%
    \reserved@a#1\@nil
  }{%
    \autoref{#3}%
  }%
}
\makeatother

\newcommand{\seq}[1]{\left(#1\right)}
\newcommand{\idx}[1]{\mbox{\underline{\sf #1}}}

\newcommand{\nand}{\mbox{\sc{nand}}}
\newcommand{\uidx}[1]{^{\langle #1 \rangle}}
\def\vec{\boldsymbol}

\DeclareMathOperator*{\ar}{\rm ar}

\DeclareMathOperator*{\zf}{\textsc{zf}}
\DeclareMathOperator*{\zfc}{\textsc{zfc}}

\DeclareMathOperator*{\consistent}{\textsc{consistent}}

\definecolor{lgreen}{rgb}{0.0, 0.48, 0.0}
\definecolor{lpurple}{rgb}{0.48, 0.0, 0.48}
\definecolor{bblue}{rgb}{0.2, 0.4, 0.8}
\hypersetup{linktocpage,
            colorlinks=true,
            linkcolor=lgreen,
            citecolor=lpurple,
            linktoc=true}

\definecolor{bblue}{rgb}{0.2, 0.4, 0.8}
\definecolor{bgreen}{rgb}{0.2, 0.6, 0.4}
\definecolor{bred}{rgb}{0.8, 0.4, 0.2}
\definecolor{bviolet}{rgb}{0.7, 0.2, 0.7}
\definecolor{blackred}{rgb}{0.6, 0.3, 0.3}
\definecolor{blackblue}{rgb}{0.3, 0.3, 0.6}

\usepackage{tikz}
\usetikzlibrary{fit,arrows,trees,shapes,shapes.geometric,calc,matrix}
\tikzset{
  treenode/.style = {align=center, inner sep=0pt, text centered,
    font=\sffamily},
  arn_nn/.style = {treenode, circle, bblue, draw=bblue,
    fill=bblue!10,
    minimum width=0.5em, minimum height=0.5em
},
  arn_n/.style = {treenode, circle, bblue, draw=bblue,
    text width=1.5em, very thick,
    fill=bblue!10},
  arn_g/.style = {treenode, circle, bgreen, draw=bgreen,
    minimum width=0.5em, minimum height=0.5em,
    fill=bblue!10},
  arn_r/.style = {treenode, circle, bred, draw=bred,
    minimum width=0.5em, minimum height=0.5em,
    fill=bviolet!10},
  arn_x/.style = {treenode, triangle, draw=black,
    minimum width=0.5em, minimum height=0.5em},
  triangle/.style = {treenode, bred, draw=bred, fill=bred!20, regular polygon, regular polygon
    sides=3, very thick, text width=1.5em },
  triangle_b/.style = {treenode, bblue, draw=bblue,
    fill=bblue!20, regular polygon, regular polygon
    sides=3, very thick, text width=1.5em },
  triangle_g/.style = {treenode, bgreen, draw=bgreen,
    fill=bgreen!20, regular polygon, regular polygon
    sides=3, very thick, text width=1.5em },
  triangle_v/.style = {treenode, bviolet, draw=bviolet,
    fill=bviolet!20, regular polygon, regular polygon
    sides=3, very thick, text width=1.5em },
  triangle_h/.style = {treenode, bblue, draw=bblue,
    fill=gray!20, regular polygon, regular polygon
    sides=3, very thick, text width=1.5em },
  arn_e/.style = {treenode, blackblue, draw=blackblue,
    fill=bblue!10, circle,
    very thick, text width=1.5em },
  arn_w/.style = {treenode, black, draw=black,
    fill=white, circle,
    densely dashed, thick, text width=1.5em }
}

\usepackage{microtype}

\DeclareMathAlphabet\mathbfcal{OMS}{cmsy}{b}{n}

\newcommand{\mynewtheorem}[2]{
  \newaliascnt{#1}{dummy}
  \newtheorem{#1}[#1]{#2}
  \aliascntresetthe{#1}
  \expandafter\def\csname #1autorefname\endcsname{#2}
}

\theoremstyle{definition}
\mynewtheorem{theorem}{Theorem}
  \mynewtheorem{proposition}{Proposition}
  \mynewtheorem{corollary}{Corollary}
  \mynewtheorem{lemma}{Lemma}

\theoremstyle{definition}
  \mynewtheorem{definition}{Definition}
  \mynewtheorem{example}{Example}
  \mynewtheorem{remark}{Remark}

\newtheorem*{problem*}{Problem}
\newtheorem*{conjecture*}{Conjecture}
\usepackage{etex}

\begin{document}

\title{A note on the asymptotic expressiveness of ZF and ZFC}
\author{Maciej Bendkowski}
\email{maciej.bendkowski@gmail.com}
\date{\today}
\thanks{The author was supported by the Polish  National  Science  Center  grant
2018/31/B/ST6/01294.}

\maketitle

\begin{abstract}
We investigate the asymptotic densities of theorems provable in Zermelo-Fraenkel
set theory $\zf$ and its extension $\zfc$ including the axiom of choice.
Assuming a canonical De Bruijn representation of formulae, we construct
asymptotically large sets of sentences unprovable within $\zf$, yet provable in
$\zfc$. Furthermore, we link the asymptotic density of $\zfc$ theorems with the
provable consistency of $\zfc$ itself. Consequently, if $\zfc$ is consistent, it is
not possible to refute the existence of the asymptotic density of $\zfc$
theorems within $\zfc$. Both these results address a recent question by Zaionc
regarding the asymptotic equivalence of $\zf$ and $\zfc$.
\end{abstract}

\section{Introduction}
In the current paper we are interested in the \emph{asymptotic expressiveness}
of first-order set theories $\zf$ and $\zfc$. More specifically, we investigate
the asymptotic density of sentences provable within these theories among all
sentences expressible in the first-order language $\mathcal{L}$ consisting of a
single binary \emph{membership} predicate $(\in)$ and no function symbols.

We start with the following problem posted recently by Zaionc.

\begin{problem*}
 Consider the theories $\zf$ and $\zfc$. What is the asymptotic density of
theorems provable within $\zfc$? Is it true that $\zfc$ is asymptotically more
expressive than $\zf$?
\end{problem*}

To make the notion of asymptotic density of theorems sound, we have to assign to
each formula $\varphi$ an integer \emph{size} $|\varphi|$ in such a way that
there exists a finite number of formulae of any given size. Having such a size
notion, we then define the
\emph{asymptotic expressiveness} of a theory $\mathcal{T}$ as the
\emph{asymptotic density} $\mu(\mathcal{T})$ of its theorems among all possible
sentences, \emph{i.e.}
\begin{equation}
 \mu(\mathcal{T}) = \lim_{n\to\infty} \frac{|\{\varphi \colon |\varphi| = n \land
\mathcal{T} \vdash \varphi \}|}{|\{\varphi \colon |\varphi| = n\}|}.
\end{equation}

In order to start addressing the above problem, we have to establish a formal
framework in which we fix certain technical, yet important details, such as the
assumed size model of formulae, or their specific combinatorial representation.
In this paper we choose to represent formulae using De~Bruijn
indices~\cite{DEBRUIJN1972381} instead of the usual notation involving named
variables. Within this setup our contributions are twofold.

Firstly, we show that $\zf$ and $\zfc$ cannot share the same asymptotic
expressiveness. Specifically, we construct an asymptotically large
(\emph{i.e.}~having positive asymptotic density) fraction of
$\mathcal{L}$\nobreakdash-sentences which, though provable in $\zfc$, cannot be proven in
the weaker system $\zf$ without the axiom of choice.
Secondly, we show that it is not possible to refute the existence of
$\mu(\zfc)$ within $\zfc$ itself. For that purpose we
link the provable existence of $\mu(\zfc)$ with the provable consistency of $\zfc$.
In light of Gödel's second incompleteness theorem, the existence of
$\mu(\zfc)$ becomes unprovable within $\zfc$.

We base our analysis on a mixture of methods from analytic combinatorics and, more
specifically, recent advances in its application in the quantitative analysis of
$\lambda$\nobreakdash-terms~\cite{benboddov2019,BODINI201845,flajolet09}.

\section{Analysis}
\subsection{Formulae representation}
Let $V$ be an infinite, denumerable set of De~Bruijn indices $\idx{0}, \idx{1},
\idx{2}, \ldots$, and $\mathcal{F}$ be a finite, functionally complete set of
proposition connectives, \emph{e.g.}~$\{\land,\lor,\neg\}$. Then, we define the
set of $\mathcal{L}$\nobreakdash-formulae $\Phi$ inductively as follows:
\begin{itemize}
\item If $\idx{n}, \idx{m} \in V$ are two indices, then $(\idx{n} \in \idx{m})$
  is a formula in $\Phi$;
\item If $\varphi_1,\ldots,\varphi_n \in \Phi$ and $\circ \in \mathcal{F}$ is an
$n$\nobreakdash-ary connective, then $\circ(\varphi_1,\ldots,\varphi_n) \in
\Phi$;
\item If $\varphi \in \Phi$, then both $(\forall \varphi)$ and $(\exists
  \varphi)$ are formulas in $\Phi$.
\end{itemize}

Recall that in the De~Bruijn notation, we replace named variables with
\emph{indices} denoting the relative distance between the represented variable
occurrence and its binding quantifier. For instance, in the De~Bruijn notation
the \emph{empty set axiom} $\exists x~\forall y~(y \not\in x)$ becomes
$\exists~\forall~(\idx{0} \not\in \idx{1})$. The index $\idx{0}$ denotes a
variable bound by the nearest quantifier, \emph{i.e.}~$\forall$. Likewise,
$\idx{1}$ denotes a variable bound by the second nearest quantifier,
\emph{i.e.}~$\exists$. In general, the index $\idx{n}$ represents a variable
occurrence which is bound by the $(n+1)$st quantifier on the path between the
index and the top node of the corresponding expression tree. If no such
quantifier exists, the index occurs \emph{free}.

\begin{remark}
  De~Bruijn introduced integer indices to facilitate the automatic manipulation
of $\lambda$\nobreakdash-terms~\cite{DEBRUIJN1972381}. From our point of view,
adopting his notation to first-order formulae presents a few important
advantages. Most notably, each sentence admits one, canonical representative.
Consequently, we do not have to concern ourselves with counting formulae
\emph{up to $\alpha$\nobreakdash-equivalence}, \emph{i.e.}~up to bound variable
names. For instance, formulae $\exists x~\forall y~(y \not\in x)$ and $\exists
y~\forall z~(z \not\in y)$ admit the \emph{same} representation
$\exists~\forall~(\idx{0} \not\in \idx{1})$. Further advantages of the De~Bruijn
notation will become clear once we start a more detailed quantitative analysis
of formulae.
\end{remark}

\subsection{Size model}
The set $\Phi$ of formulae we consider can be neatly encapsulated in the following, more
symbolic specification:
\begin{align}\label{eq:spec}
  \Phi_\infty &:=~V \in V~|~\forall \Phi_\infty~|~\exists \Phi_\infty
        ~|~\bigcup_{(\circ) \in \mathcal{F}} \circ(\vec{\Phi}_\infty).
\end{align}
Here $V$ denotes the class of De~Bruijn indices whereas the boldface notation
$\vec{\Phi}_\infty$ denotes vectors of lengths matching the arities of respective
connectives. So, for instance, if $\mathcal{F} = \{ \nand \}$ then the final 
alternative in~\eqref{eq:spec} becomes $\nand (\Phi_\infty, \Phi_\infty)$.

Given such a general combinatorial specification of formulae, we assume a
\emph{unary} representation of indices. In other words, we represent the index
$\idx{n}$ as an $n$\nobreakdash-fold successor of zero, \emph{i.e.}~$S^{(n)} 0$.
In doing so, the class $V$ admits a simple, recursive definition:
\begin{align}\label{eq:spec:idx}
 V &:= 0~|~S V. 
\end{align}

Note that combined,~\eqref{eq:spec} and~\eqref{eq:spec:idx} constitute a simple
algebraic specification of formulae. To complete the above size model, we assume
a so-called \emph{natural} size model, \emph{cf.}~\cite{10.1093/logcom/exx018}, which adheres
with the previously mentioned finiteness condition. In this model, each
constructor, \emph{i.e.}~successor $S$, zero $0$, membership predicate $(\in)$,
quantifiers $\forall$, $\exists$, and connectives $(\circ)$, is assigned
\emph{weight} one. The \emph{size} of a formula becomes then the total
weight of all constructors it is built from. So, for instance,
$\exists~\forall~(\idx{0} \not\in \idx{1})$ has size $7$ as it consists of seven
constructors. Note that $\not\in$ is a shorthand for two constructors $(\in)$
and $(\neg)$, and that $\idx{1} \equiv S~0$ consists of two constructors.

\begin{remark}
 We choose the natural size notion for technical simplicity. Let us mention that
our analysis extends onto more sophisticated weighing systems, though it is specific
to the unary representation of indices and does not immediately apply to other
representations, \emph{e.g.}~involving a more compact binary encoding of indices, or
counting formulae with named variables up to $\alpha$\nobreakdash-equivalence.
\end{remark}

\subsection{Counting formulae}
Given the algebraic specification~\eqref{eq:spec} of formulae, we can easily
lift it onto the level of generating functions using so-called symbolic methods,
see~\cite[Part A. Symbolic Methods]{flajolet09}. Let $\Phi_\infty(z) = \sum_{n
\geq 0} a_n z^n$ be the generating function in which the $n$th coefficient $a_n
= [z^n]\Phi_\infty(z)$ denotes the number of formulae of size $n$.

Then, based on~\eqref{eq:spec} $\Phi_\infty(z)$ satisfies the relation
\begin{align}\label{eq:spec:2}
  \Phi_\infty(z) &= z {\left( \frac{z}{1-z} \right)}^2 +
                   2 z \Phi_\infty(z) +
        \sum_{(\circ) \in \mathcal{F}} z {\Phi_\infty(z)}^{\ar(\circ)}.
\end{align}
Here $\ar(\cdot)$ denotes the the arity of the respective symbol.

\begin{proposition}\label{prop:phi:infty:puiseux}
  The generating function $\Phi_\infty(z)$ admits a Puiseux expansion in form
of
  \begin{equation}\label{eq:phi:puisseux}
    \Phi_\infty(z) = a - b \sqrt{1 - \frac{z}{\rho}} +
    O\left(\big|1-\frac{z}{\rho}\big|\right)
  \end{equation}
  where $0 < \rho < 1$, and $a,b > 0$.
\end{proposition}

\begin{proof}
  We apply the Drmota--Lalley--Woods theorem, see \emph{e.g.}~\cite[Theorem VII.6]{flajolet09}.

  Recall that the set $\mathcal{F}$ of connectives is functionally complete and
so by Post's theorem it must contain a connective $(\circ)$ of arity $n \geq 2$,
\emph{cf.}~\cite{pelletier1990}. It means that~\eqref{eq:spec:2} is non-linear
in $\Phi_\infty(z)$. By construction, it is also algebraic positive and
algebraic irreducible. Algebraic properness, sometimes referred to as
\emph{well-foundedness}, follows for instance from Joyal and Labelle's implicit
species theorem, \emph{cf.}~\cite{PIVOTEAU20121711}. Indeed, consider the
Jacobian $\mathcal{H}(z, \Phi_\infty)$ associated with $\Phi_\infty$:
\begin{equation}
  \mathcal{H}(z, \Phi_\infty) =
  2z + \sum_{(\circ) \in \mathcal{F}} z \ar(\circ) {\Phi_\infty(z)}^{\ar(\circ) - 1}.
\end{equation}
Note that $\mathcal{H}(0, 0) = 0$ and so, as a $1 \times 1$ matrix, it is
trivially nilpotent. Algebraic aperiodicity is a consequence of the fact that
for all sufficiently large $n$ there exists a formula of size $n$ (for instance
the sole index $\idx{n-1}$). Hence $[z^n]\Phi_\infty(z) \neq 0$.

The generating function $\Phi_\infty(z)$ meets therefore all necessary requirements.
We can apply the Drmota--Lalley-Woods theorem and conclude
that~\eqref{eq:spec:2} has a unique dominant singularity $\rho$ and a suitable
Puiseux expansion of declared form.
\end{proof}

\begin{remark}
  Given the Puiseux expansion of $\Phi_\infty(z)$ we can readily use transfer
  theorems~\cite[Section VI.3]{flajolet09} to obtain an asymptotic estimate for
  $[z^n]\Phi_\infty(z)$ standing for the number of formulae of size $n$:
  \begin{equation}\label{eq:all:estimate}
    [z^n]\Phi_\infty(z) \sim C \cdot \rho^n n^{-3/2}.
  \end{equation}
\end{remark}

  \subsection{Counting sentences}
Since we are interested in the asymptotic density of theorems, we need to
establish asymptotic estimates for the number of \emph{sentences},
\emph{i.e.}~formulae without free indices. We follow a method similar
to~\cite{benboddov2019,BODINI201845} developed for the purpose of counting
closed $\lambda$\nobreakdash-terms in the De~Bruijn representation.

Let us start with introducing \emph{$m$\nobreakdash-open formulae}. Analogously
to $m$\nobreakdash-open $\lambda$\nobreakdash-terms, we call a
formula $\varphi$ $m$\nobreakdash-open if prepending $\varphi$ with $m$ head
quantifiers, be it universal or existential, turns $\varphi$ into a sentence.
For instance, the formula $\left( \forall~\idx{0} \to (\exists~\idx{0} \land
  \idx{2}) \right)$ is
$1$\nobreakdash-open as $\idx{2}$ occurs free, however becomes bound once we
introduce a single head quantifier. Note that by such a definition,
$m$\nobreakdash-open formulae are at the same time $(m+1)$\nobreakdash-open.
In particular, sentences are just $0$\nobreakdash-open (closed) formulae.

By symbolic methods, the definition of $m$\nobreakdash-open formulae gives rise
to an infinite specification involving all the classes of
$m$\nobreakdash-open formulae $\seq{\Phi_m}_{m \geq 0}$:
\begin{align}\label{eq:m-open:spec}
  \begin{split}
    \Phi_0 &:= \forall \Phi_1~|~\exists \Phi_1
    ~|~\bigcup_{(\circ) \in \mathcal{F}} \circ(\vec{\Phi}_0)\\
    \Phi_1 &:= \forall \Phi_2~|~\exists \Phi_2
    ~|~\bigcup_{(\circ) \in \mathcal{F}} \circ(\vec{\Phi}_1)
    ~|~V_{< 1} \in V_{< 1}\\
    \Phi_2 &:= \forall \Phi_3~|~\exists \Phi_3
    ~|~\bigcup_{(\circ) \in \mathcal{F}} \circ(\vec{\Phi}_2)
    ~|~V_{< 2} \in V_{< 2}\\
    &\ldots \\
    \Phi_m &:= \forall \Phi_{m+1}~|~\exists \Phi_{m+1}
    ~|~\bigcup_{(\circ) \in \mathcal{F}} \circ(\vec{\Phi}_m)
    ~|~V_{< m} \in V_{< m}\\
    &\ldots
  \end{split}
\end{align}
In words, a formula $\varphi$ is $m$\nobreakdash-open if it is in form of
$(\forall \tau)$ or $(\exists \tau)$ where $\tau$ is $(m+1)$\nobreakdash-open, or in
form of $\varphi = \circ(\tau_1,\ldots,\tau_n)$ where $\tau_1,\ldots,\tau_n$ are
again $m$\nobreakdash-open, or finally, if $\varphi$ is in form of $(\idx{n} \in
\idx{k})$ where both $\idx{n}$ and $\idx{k}$ are two of the $m$ initial indices
$V_{< m} = \{\idx{0},\ldots,\idx{m-1}\}$.

Note that $\Phi_0 \subsetneq \Phi_1 \subsetneq \cdots \subsetneq \Phi_m
\subsetneq \cdots \subsetneq \Phi_\infty$. Consecutive classes subsume and
extend all the previous ones so $\Phi_m$ resembles more and more $\Phi_\infty$
as $m \to \infty$. Let us formalise this intuition and show that the infinite
$\seq{\Phi_m}_{m \geq 0}$ is a \emph{forward recursive system} in the sense of
the following definition, \emph{cf.}~\cite[Definition 5.5]{benboddov2019}\footnote{The current
definition is a simplified version of~\cite[Definition 5.5]{benboddov2019}. The
original one involves systems of functional equations and permits additional vectors of so-called \emph{marking variables}
which can be used, for instance, to track the behaviour of certain interesting
sub-patterns of random structures.}.

\begin{definition}[Forward recursive systems]\label{definition:infinitely:nested:systems}
  Let $z$ be a formal variable. Consider the infinite sequences
  $\seq{L \uidx m}_{m \geq 0}$ and $\seq{\mathcal K \uidx m}_{m \geq 0}$ of formal power
  series $L \uidx m(z)$ and $\mathcal K \uidx m(\ell_1, \ell_2, z)$.
  Assume that $\seq{L \uidx m}_{m \geq 0}$ and
  $\seq{\mathcal K \uidx m}_{m \geq 0}$ satisfy
  \begin{equation}\label{eq:infinite:system:forward:recursive:system}
    L \uidx m = \mathcal K \uidx m \left(L \uidx m, L \uidx
      {m+1}, z \right).
  \end{equation}
  Then, we say that the
system~\eqref{eq:infinite:system:forward:recursive:system} is \emph{forward
  recursive}.

Furthermore, consider a \emph{limiting system} in form of
\begin{equation}\label{eq:infinite:system:limit:system}
  L \uidx \infty = \mathcal K \uidx \infty \left(L \uidx \infty, L \uidx
    {\infty}, z \right)
\end{equation}
where $L \uidx \infty(z)$ and $\mathcal{K} \uidx \infty(\ell_1, \ell_2, z)$
are formal power series, and moreover
$K \uidx \infty$ is analytic at
$\left(\ell_1, \ell_2, z\right) =
\left(0,0,0\right)$. In this setting, we say that the
system~\eqref{eq:infinite:system:forward:recursive:system}:
    \begin{enumerate}
        \item is \emph{infinitely nested} if
            $\mathcal K \uidx m (\ell_1, \ell_2, z)
        \preceq
        \mathcal K \uidx \infty (\ell_1, \ell_2, z)$
        meaning that for all $n \geq 0$
        \begin{equation*}
          [z^n]\mathcal K \uidx m (\ell_1, \ell_2, z)
          \leq
          [z^n]\mathcal K \uidx \infty (\ell_1, \ell_2, z);
        \end{equation*}
        \item \emph{tends to an irreducible context-free schema} if it is
            infinitely nested and its corresponding limiting
            system~\eqref{eq:infinite:system:limit:system} satisfies the
            premises of the Drmota--Lalley--Woods theorem~\cite[Theorem
            VII.6]{flajolet09}, \emph{i.e.}~is a
            polynomial, non-linear functional equation which is
            algebraic positive, proper, irreducible and aperiodic;
        \item
            is \emph{exponentially converging} if
            there exists a formal power series 
            $A(z)$ and \( B(z) \) such that
        \begin{equation}
            \mathcal K \uidx \infty
            (L \uidx \infty, L \uidx \infty, z)
            -
            \mathcal K \uidx m
            (L \uidx \infty, L \uidx \infty, z)
            \preceq
            A(z) \cdot {B(z)}^m,
          \end{equation}
     and both $A(z)$
    and \( B(z) \) are
        analytic in the disk
        \( |z| < \rho + \varepsilon \) for some \( \varepsilon > 0 \)
        where \( \rho \) is the dominant singularity of the
        limit system~\eqref{eq:infinite:system:limit:system}.
        Moreover, we have \( \left|B(\rho)\right| < 1 \).
    \end{enumerate}
 \end{definition}
    
\begin{proposition}
  The infinite system $\seq{\Phi_m}_{m \geq 0}$ is an infinitely
  nested, forward recursive system which tends to the irreducible context-free
  schema $\Phi_\infty$ at an exponential convergence rate.
\end{proposition}

\begin{proof}
  Let us unpack the all the definitions one at a time. First, since $\Phi_m$
involves $\Phi_{m+1}$ in its specification~\eqref{eq:m-open:spec}, note that the
infinite system $\seq{\Phi_m}_{m \geq 0}$ fits the definition of a forward
recursive system assuming $\Phi_\infty$ as its limiting system.

  Next, let us consider $\Phi_m$ and the limiting system $\Phi_\infty$. Note that each
$m$\nobreakdash-open formula is accounted for in $[z^n]\Phi_m(z)$ and in
$[z^n]\Phi_\infty(z)$. Hence $\Phi_m(z) \preceq \Phi_\infty(z)$ and
$\seq{\Phi_m}_{m \geq 0}$ is indeed infinitely nested. By the arguments
presented in the proof of~\cref{prop:phi:infty:puiseux} it is at the same time
tending to the irreducible context-free schema $\Phi_\infty$.

Finally, note that by~\eqref{eq:m-open:spec} $\Phi_m(z)$ satisfies the equation
\begin{equation}\label{eq:spec:m}
  \Phi_m(z) =
  z \left(\frac{z\left(1-z^m\right)}{1-z}\right)^2 +
  2 z \Phi_{m+1}(z) +
  \sum_{(\circ) \in \mathcal{F}} z {\Phi_m(z)}^{\ar(\circ)}.
\end{equation}
Therefore
\begin{align}
  \begin{split}
    \Phi_\infty(\Phi_\infty,\Phi_\infty,z) -
    \Phi_m(\Phi_\infty,\Phi_\infty,z) &=
    z {\left( \frac{z}{1-z} \right)}^2 - 
    z \left(\frac{z\left(1-z^m\right)}{1-z}\right)^2\\
    &= \frac{z^3}{\left(1-z\right)^2}\left(2 z^m - z^{2m}\right) \preceq \frac{2 z^3}{\left(1-z\right)^2} z^{m}.
  \end{split}
\end{align}
In this form, it is clear that $\seq{\Phi_m(z)}_{m \geq 0}$ is exponentially
converging.
\end{proof}

Having established the behaviour of $\seq{\Phi_m}_{m \geq 0}$
we are ready to apply the following general result~\cite[Theorem
5.9]{benboddov2019} tailored here so to fit our specific application.

\begin{theorem}
    Let $\mathcal{S}$ be an infinitely nested, forward recursive
    system which tends to an
    irreducible context-free schema at an exponential convergence rate.  Then,
    the respective solutions \( L\uidx m (z) \) of $\mathcal{S}$ admit
    for each \( m \geq 0 \) an asymptotic expansion of their coefficients
    as \( n \to \infty \) in form of
    \begin{equation}\label{eq:prop:infinite:sys:puiseux}
    [z^n]L \uidx m(z) \sim [z^n] \sum_{k \geq 0} c_{k} \uidx m \left(
        1 - \dfrac{z}{\rho}
    \right)^{k/2}
\end{equation}
where \( \rho \) is the dominant singularity of the corresponding
limiting system~\eqref{eq:infinite:system:limit:system}.
\end{theorem}

As for general formulae, a direct application of transfer theorems gives us the
following asymptotic estimate for the number of $m$\nobreakdash-open formulae.

\begin{proposition}\label{prop:sentence:counting}
  The number $[z^n]\Phi_m(z)$ of $m$\nobreakdash-open formulae of size $n$
  satisfies the following asymptotic estimate:
    \begin{equation}\label{eq:mopen:estimate}
        [z^n]\Phi_m(z) \sim C_m \cdot \rho^n n^{-3/2}.
      \end{equation}
\end{proposition}

  Note that estimates~\eqref{eq:mopen:estimate} share the same exponential
$\rho^n$ and sub-exponential $n^{-3/2}$ factors. These are also the same for the
number of all formulae, \emph{cf.}~\eqref{eq:all:estimate}. What differentiates
them are the respective multiplicative constants $C_m$ and $C$.

\section{Applications}
Given the asymptotic estimate for the number of $m$\nobreakdash-open
formulae, we continue our investigations into the asymptotic
expressiveness of $\zf$ and $\zfc$. Let us introduce the following useful
concept of \emph{formulae templates}.

\begin{definition}[Formulae templates]
 A \emph{template} $\mathcal{C}$ is a formula with a single hole $[\cdot]$
instead of some sub-formula in form of $(\idx{n} \in \idx{m})$. To denote the result of
substituting a formula $\varphi$ for $[\cdot]$ in $\mathcal{C}$ we write
$\mathcal{C}[\varphi]$. Note that the outcome of such a substitution is always a
valid formula.

We call a template \emph{$m$\nobreakdash-permissive} if for each
$m$\nobreakdash-open formula $\varphi$ the resulting $\mathcal{C}[\varphi]$ is
a sentence, \emph{i.e.}~is $0$\nobreakdash-open. By analogy to formulae, the
\emph{size} of a template is the total weight of its building constructors,
assuming that $[\cdot]$ weights zero.
 
For instance, consider the template $\mathcal{C} = \exists~\forall~[\cdot]$ of
size two. The result of substituting $(\idx{0} \not\in \idx{1})$ into
$\mathcal{C}$ is the formula $\mathcal{C}[(\idx{0} \not\in \idx{1})] =
\exists~\forall~(\idx{0} \not\in \idx{1})$. The hole $[\cdot]$ is proceeded by
two quantifiers in $\mathcal{C}$ so $\mathcal{C}$ is $2$\nobreakdash-permissive.
Note that, in general, if $\mathcal{C}$ is $m$\nobreakdash-permissive, then it
is also $0$\nobreakdash-permissive, $1$\nobreakdash-permissive, \emph{etc}.
\end{definition}

\begin{lemma}\label{lem:context}
  Let $\mathcal{C}$ be an $m$\nobreakdash-permissive template and
  $\mathcal{L}(\mathcal{C}) = \{ \mathcal{C}[\varphi] \colon \varphi \text{ is }
  m\text{\nobreakdash-open}\}$. Then, the set $\mathcal{L}(\mathcal{C})$ has
  positive asymptotic density in the set of all sentences.
\end{lemma}

\begin{proof}
  Assume that the template $\mathcal{C}$ has size $d$. It means that the
sentence $\mathcal{C}[\varphi]$ is of size $d + |\varphi|$. Let us estimate the
number of formulae of this form. By~\cref{prop:sentence:counting} the number of
sentences $\phi \in \mathcal{L}(\mathcal{C})$ of size $n$ satisfies the
asymptotic estimate $C_m\cdot
\rho^{n-d} {\left(n-d\right)}^{-3/2}$. Likewise, the number of all sentences of size $n$
is estimated by $C_0\cdot \rho^{n} {n}^{-3/2}$. Hence, the asymptotic density
of $\mathcal{L}(\mathcal{C})$ in the set of all
sentences admits the following estimate:
 \begin{equation}
   \mu(\mathcal{L}(\mathcal{C})) \sim \frac{C_m \cdot \rho^{n-d} {(n-d)}^{-3/2}}{C_0 \cdot \rho^n n^{-3/2}}
        \sim \frac{C_m}{C_0 \cdot \rho^d} > 0.
      \end{equation}
\end{proof}

 The above lemma allows us to construct asymptotically large
 (\emph{i.e.}~having positive asymptotic density) sets of sentences whose
structure fits the imposed template pattern. We are going to exploit this
construction in the following sections.

\subsection{Consistency, extensions, and asymptotic expressiveness}
To investigate the asymptotic expressiveness of both $\zf$ and $\zfc$ we start
with several propositions regarding general, abstract axiomatic set theories and
their properties. For convenience, we use $\mu^{-}(\mathcal{T})$
and $\mu^{+}(\mathcal{T})$ to denote $\displaystyle\liminf_{n \to \infty}$ and
$\displaystyle\limsup_{n \to \infty}$ of
$\dfrac{|\{\varphi \colon |\varphi| = n \land \mathcal{T} \vdash \varphi \}|}{|\{\varphi \colon |\varphi| = n\}|}$,
respectively.

\begin{proposition}\label{prop:non-trivial}
    Let $\mathcal{T}$ be a consistent axiomatic system. Then,
    the set of $\mathcal{T}$-theorems cannot have a trivial asymptotic density,
    \emph{i.e.}~$0 < \mu^{-}(\mathcal{T})$ and $\mu^{+}(\mathcal{T}) < 1$.
\end{proposition}

\begin{proof}
  Let $\tau$ be an arbitrary tautology. Consider the following templates:
  \begin{equation}
   \mathcal{C} = ([\cdot] \lor \tau) \qquad \text{and} \qquad
\overline{\mathcal{C}} = ([\cdot] \land \neg \tau).
  \end{equation}

  Note that $\mathcal{L}(\mathcal{C}) = \{ \mathcal{C}[\varphi] \colon \varphi
  \in \Phi_0\}$ consists of tautologies (in particular, $\mathcal{T}$-theorems).
  Likewise, $\overline{\mathcal{L}}(\mathcal{C}) = \{ \overline{\mathcal{C}}[\varphi] \colon \varphi
  \in \Phi_0\}$ consists of anti-tautologies. By~\cref{lem:context} both have
  positive asymptotic density in the set of all sentences. Therefore,
$0 < \mu^{-}(\mathcal{T})$ and $\mu^{+}(\mathcal{T}) < 1$.
\end{proof}

In other words, consistent theories cannot have a trivial asymptotic
expressiveness. In particular, this remark applies to $\zf$ and $\zfc$
(assuming, of course, their consistency). The
following slight refinement gives an \emph{if and only if} condition linking
inconsistent theories and trivial asymptotic expressiveness.

\begin{proposition}\label{prop:inconsistent}
  An axiomatic system $\mathcal{T}$ is inconsistent if and only if $\mu(\mathcal{T}) = 1$.
\end{proposition}

\begin{proof}
  Assume that $\mathcal{T}$ is inconsistent. It means that we can derive a
contradiction within $\mathcal{T}$. Since \emph{ex falso quodilibet}, any
sentence $\varphi$ is a theorem of $\mathcal{T}$. Trivially, it means that
$\mu(\mathcal{T})$ exists and is equal to one. Now, let us assume that
$\mathcal{T}$ is consistent.  Note that $\mathcal{T}$ cannot have an asymptotic
expressiveness one as for consistent theories we can construct asymptotically
large sets of anti-tautologies witnessing $\mu^{+}(\mathcal{T}) < 1$,
\emph{cf.}~\cref{prop:non-trivial}.
\end{proof}

Template formulae allow us to derive also the following result.

\begin{proposition}
  Let $\mathcal{T}$ be a consistent axiomatic system and $\phi$ be an sentence
independent of $\mathcal{T}$, \emph{i.e.}~$\mathcal{T} \not\vdash \phi$ nor
$\mathcal{T} \not\vdash \neg \phi$. Then, there exists a set of theorems of
the extended $\mathcal{T} + \phi$ which has positive asymptotic density.
\end{proposition}

\begin{proof}
  Let $\tau$ be an arbitrary tautology. Consider the context $\mathcal{C} =
\left( \tau \lor [\cdot] \right) \to \phi$ and the set $\mathcal{L}(\mathcal{C})
= \{\mathcal{C}[\varphi] \colon \varphi \in \Phi_0\}$ it generates. Note that
by~\cref{lem:context} the set $\mathcal{L}(\mathcal{C})$ is asymptotically
large. Let us choose an arbitrary sentence $\xi \in \mathcal{L}(\mathcal{C})$.
Note that $\mathcal{T} + \phi \vdash \xi$ as we can simply assume $\xi$'s
premise and discharge the axiom $\phi$. Moreover, note that $\mathcal{T}
\not\vdash \xi$. Indeed, suppose to the contrary that $\mathcal{T} \vdash \xi$.
Since $\xi$'s premise is a tautology it has a proof in $\mathcal{T}$. By
\emph{modus ponens} we conclude that $\mathcal{T} \vdash \phi$ which contradicts the
fact that $\phi$ is independent of $\mathcal{T}$.
\end{proof}

\begin{remark}
  The axiom of choice is a prominent example of a sentence independent of
Zermelo-Fraenkel's set theory $\zf$, \emph{cf.}~\cite{kunen_2011}. Consequently,
it is possible to construct asymptotically large sets of
$\zfc$\nobreakdash-theorems which cannot be proven in the weaker system $\zf$.

  Note that our argument holds for \emph{any} consistent axiomatic set theory
  and independent sentences. Moreover, for sufficiently rich theories Gödel's
  first incompleteness theorem, see~\cite{smith_2013}, guarantees that such an
  argumentation can be carried out \emph{ad infinitum}.
\end{remark}

\subsection{Unprovable existence of asymptotic expressiveness}
Having discussed the properties of general axiomatic theories, let us now
concentrate on $\zfc$. The following result is a simple
consequence of Gödel's second incompleteness theorem, \emph{cf.}~\cite{smith_2013}.

\begin{proposition}
  Let $\mu$ be a predicate definable in $\zfc$ such that $\zfc \vdash \mu(g)$ if
  and only if $\mu(\zfc)$ exists and is equal to $g$, and $\consistent$
  be the canonical predicate encoding the consistency of $\zfc$. Assume that $\zfc$ is consistent. If
  \begin{equation}\label{eq:zfc:equivalence}
    \zfc \vdash
    \consistent \longleftrightarrow \neg \mu(1),
  \end{equation}
  then
  $\zfc \not\vdash \neg\exists~g \colon \mu(g)$.
\end{proposition}

\begin{proof}
  Suppose that $\neg \exists~g \colon \mu(g)$ can be proven in $\zfc$,
\emph{i.e.}~$\zfc \vdash \neg \exists~g \colon \mu(g)$. Equivalently, $\zfc
\vdash \forall~g \colon \neg \mu(g)$. In particular $\zfc \vdash \neg \mu(1)$.
By~\eqref{eq:zfc:equivalence} it holds $\zfc \vdash \consistent$. Gödel's second
incompleteness theorem, as instantiated for $\zfc$, provides the required contradiction.
\end{proof}

\begin{remark}
  Note that the assumption~\eqref{eq:zfc:equivalence} states
that~\cref{prop:inconsistent} can be formalised in $\zfc$. We can
safely assume that this laborious task can be accomplished. It is unclear, however, if the same
results hold for the weaker theory $\zf$. In the proof
of~\cref{prop:inconsistent} we use quite deep results from analytic
combinatorics whose use of the axiom of choice is, as far as we know,
undetermined.
\end{remark}

\section{Conclusions}
We investigated the asymptotic expressiveness of $\zf$ and $\zfc$ developing a
general argument that within $\zfc$ it is possible to prove an asymptotically
large set of theorems unprovable in $\zf$ alone. By linking provable consistency
of $\zfc$ and its asymptotic expressiveness, we argued that within $\zfc$ it is
not possible to disprove the existence of $\zfc$'s asymptotic expressiveness.
Moreover, by the same argument, it is not possible to prove the existence of a
non-trivial asymptotic expressiveness of $\zfc$ within itself.

Let us note that $\zfc \not \vdash \consistent$ along with the assumption that
$\zfc$ is consistent imply that $\zfc$ has a model $\mathcal{M}$ such that
$\mathcal{M} \vDash \mu(1)$. This model witnesses the fact that
$\zfc \not\vdash \neg\exists~g \colon \mu(g)$.  Nevertheless, it is not clear
whether $\zfc \not\vdash \exists~g \colon \mu(g)$ holds, \emph{i.e.}~if there
exists a model of $\zfc$ in which theorems of $\zfc$ do not have an asymptotic
density. We speculate that establishing a model $\mathcal{M}$ such that
$\mathcal{M} \not\vDash \exists~\mu(g)$ would require a clever mixture of
forcing, and analytic methods.

\bibliography{references}{}
\bibliographystyle{plain}
\end{document}